\documentclass[a4paper, 11pt]{amsart}

\usepackage{lmodern}
\usepackage[english]{babel}
\usepackage[utf8]{inputenc}
\usepackage{amssymb}
\usepackage{amsfonts}
\usepackage{amsthm}
\usepackage{amsmath}

\usepackage{mathrsfs, soul}

\usepackage{hyperref}
\usepackage{enumitem}
\usepackage{bm}
\usepackage{tikz-qtree}

\usepackage[margin=1.1in]{geometry}
\setlist[itemize]{leftmargin=*}

\usepackage[normalem]{ulem}

\theoremstyle{plain}
\newtheorem{theorem}{Theorem}[section]
\newtheorem{corollary}[theorem]{Corollary}
\newtheorem{definition}[theorem]{Definition}
\newtheorem{lemma}[theorem]{Lemma}
\newtheorem{proposition}[theorem]{Proposition}

\theoremstyle{definition}

\theoremstyle{remark}
\newtheorem*{remark}{Remark}

\numberwithin{equation}{section}

\newcommand\numberthis{\stepcounter{equation}\tag{\theequation}}

\newcommand{\e}{{\rm e}}
\newcommand{\df}{\mathop{}\!\mathrm{d}}
\newcommand{\eps}{{\varepsilon}}

\newcommand{\C}{{\mathbb C}}

\newcommand{\N}{{\mathbb N}}
\newcommand{\Q}{{\mathbb Q}}

\newcommand{\R}{{\mathbb R}}

\newcommand{\1}{{\mathbf 1}}

\newcommand{\CC}{{\mathcal C}}

\newcommand{\vphi}{{\bm \phi}}

\newcommand{\vt}{{\bm t}}

\newcommand{\cE}{{\mathcal E}}

\newcommand{\I}{{\mathcal I}}

\newcommand{\cV}{{\mathcal V}}

\DeclareMathOperator{\G}{{\mathcal G}}

\DeclareMathOperator{\Res}{Res}

\DeclareMathOperator{\sgn}{sgn}

\newcommand{\abs}[1]{{\left| {#1} \right|}}
\newcommand{\norm}[1]{{\left\| {#1} \right\|}}
\newcommand{\floor}[1]{{\left\lfloor {#1} \right\rfloor}}

\renewcommand\Re{\operatorname{Re}}
\renewcommand\Im{\operatorname{Im}}

\newcommand{\nvt}{\norm{\vt}}

\numberwithin{equation}{section}

\title[Oscillatory integrals related to infinitely divisible distributions]{Effective estimation of some oscillatory integrals related to infinitely divisible distributions}
\date{\today}

\author{S. Bettin}
\address{SB: Dipartimento di Matematica, Universit\`a di Genova, via Dodecaneso 35, 16146 Genova, Italy}
\email{bettin@dima.unige.it}

\author{S. Drappeau}
\address{SD: Aix Marseille Universit\'e, CNRS, Centrale Marseille, I2M UMR 7373, 13453 Marseille, France}
\email{sary-aurelien.drappeau@univ-amu.fr}

\begin{document}

\begin{abstract}
  We present a practical framework to prove, in a simple way, two-terms asymptotic expansions for Fourier integrals
  $$ \I(t) = \int_\R(\e^{it\phi(x)}-1) \df \mu(x) $$
  where~$\mu$ is a probability measure on~$\R$ and~$\phi$ is measurable. This applies to many basic cases, in link with Levy's continuity theorem. We present applications to limit laws related to rational continued fractions coefficients.
\end{abstract}



\thanks{The authors thank the referee for his or her remarks and a careful reading of the paper. This paper was partially written during a visit of of S. Bettin at the Aix-Marseille University and a visit of S. Drappeau at the University of Genova. The authors thank both Institution for the hospitality and Aix-Marseille University and INdAM for the financial support for these visits. S. Bettin is member of the INdAM group GNAMPA and his work is partially supported by PRIN 2017 ``Geometric, algebraic and analytic methods in arithmetic''.}

\maketitle

\section{Introduction}

Let~$\mu$ be a probability measure on~$\R$, and~$\phi:\R \to \R$ be~$\mu$-measurable. The present paper is concerned with asymptotic formulæ for the Fourier integrals associated with~$\phi$ near the origin,
\begin{equation}
  \I[\phi](t) := \int(\e^{it\phi(x)}-1) \df \mu(x), \quad (t\to 0). \label{eq:integral-question-Iphi}
\end{equation}
Such estimates are connected with the question of whether the push-forward measure~$\phi_*(\mu)$ belongs to the bassin of attraction of a stable law, see Chapter~2 of~\cite{IbragimovLinnik1971}.
Our interest in this question originates from this point of view, and more specifically from the work~\cite{BettinDrappeau} where we study the convergence towards stable laws of the value distribution of invariants related to modular forms. In the setting of~\cite{BettinDrappeau}, the measure~$\mu$ is the Gauss-Kuzmin distribution
$$ \df \mu(x) = \frac{\df x}{(1+x)\log 2} \qquad (x\in[0, 1]), $$
and this measure is invariant under the Gauss map~$T(x) = \{1/x\}$, where~$\{x\} = x - \floor{x}$ is the fractional part of~$x$. More precisely, in~\cite{BettinDrappeau}, we are interested in Birkhoff sums
\begin{equation}
  \sum_{j=1}^r \phi(T^r(x)), \qquad (T^r = T \circ \dotsb \circ T),\label{eq:rational-birkhoff}
\end{equation}
where~$x$ varies among rationals and~$r\geq 0$ is the length of the continued fractions expansion of~$x$. In the set of rationals we consider, these sums are found to typically behave as sums of the shape
$$ \sum_{j=1}^r \phi(X_r) $$
where~$(X_j)_{1\leq j \leq r}$ are i.i.d. random variables distributed according to the Gauss-Kuzmin measure~$\mu$. Then effective estimates for the integral~\eqref{eq:integral-question-Iphi}, in conjunction with~\cite[Theorem~3.1]{BettinDrappeau} and the Berry-Esseen inequality~\cite[equation~(XVI.3.13)]{Feller1971} are used to obtain uniform limit theorems for the rational Birkhoff sums~\eqref{eq:rational-birkhoff}.

We return to the setting where~$\mu$ is an arbitrary probability measure on~$\R$.
Integrals~\eqref{eq:integral-question-Iphi} are related to the methods of asymptotic analysis mentioned \emph{e.g.} in Chapter~9 of the monograph~\cite{Olver1997}. When expressed as convolution integrals~$\int_x h(tx) f(x) \df x$, they are refered to as $h$-transforms in~\cite{BleisteinHandelsman1986}, and are also the topic of interest of the recent work~\cite{Lopez2008}. The variety in assumptions and methods seems to prevent us from having a uniform framework for estimating~\eqref{eq:integral-question-Iphi}.

The goal of the present paper is to present and prove several basic estimates through which one can give a streamlined and simple proof of an effective asymptotic expansion of the integral~\eqref{eq:integral-question-Iphi}, including the terms of interest in central limit theorems.

\begin{definition}
  Given~$\alpha\in(0, 3]$ and two positive functions~$L, R$ defined in a neighborhood of~$0$ in~$\R_+^*$, we denote by~$\G(\alpha, L, R)$ the set of functions~$\phi:\R\to\R$ such that for some numbers~$c_1, c_2 \in \R$ and~$c_* \in \C$, and all small enough~$t>0$, there holds
  \begin{equation}
    \I[\phi](t) = i c_1 t + c_2 t^2 + c_* t^\alpha L(t) + O(t^3 + t^\alpha R(t)).\label{eq:expansion-Iphi}
  \end{equation}
\end{definition}
\begin{remark}
  \begin{itemize}
  \item If $R=O(t^\eps)$ for any $\eps>0$ and~$\alpha<1$, the term~$c_1t$ in~\eqref{eq:expansion-Iphi} is part of the error term, and likewise for~$c_2 t^2$ if~$\alpha<2$.
  \item We will be interested in the largest one or two terms in the expansion~\eqref{eq:expansion-Iphi}. The case~$\alpha=3$, $L = R \equiv 1$ corresponds to an order~$2$ Taylor expansion.
  \item Whenever the expansion~\eqref{eq:expansion-Iphi} holds for~$\phi$, we will denote the coefficients by~$c_1(\phi)$, $c_2(\phi)$, $c_\ast(\phi)$ respectively.
  \end{itemize}
\end{remark}

\begin{theorem}
  \begin{enumerate}
  \item If~$\int \abs{\phi(x)}^\alpha \df\mu(x) < \infty$ for some~$\alpha \in (0, 3]$, then~$\phi \in \G(\alpha, 1, 1)$.
  \item Suppose that~$\df \mu = f \df \nu$ where~$\nu$ is the Lebesgue measure and~$f\in\CC^1([0, 1])$. Then for all~$a\in \R^*$,~$\beta>3$ and~$\lambda\geq 0$, the function
    $$ \phi : (0, 1]\to \R, \qquad \phi(x) = a x^{-\beta} \abs{\log x}^\lambda, $$
    belongs to~$\G(\frac1\beta, \abs{\log}^{\lambda/\beta+v}, \abs{\log}^{\lambda/\beta+v-1+\eps})$ for any~$\eps \in (0, 1]$, where~$v=1$ for~$\beta \in \{1/2, 1\}$ and~$v=0$ otherwise.
  \item Given two measurable functions~$\phi_1, \phi_2$, such that~$\phi_j \in \G(\alpha_j, L_j, R_j)$ with~$t^{\alpha_2} L_2(t) = O(t^{\alpha_1}L_1(t))$ as~$t\to 0$, then~$\phi_1+\phi_2 \in \G(\alpha_1, L_1, R_+)$ for some positive function~$R_+$ explicit in terms of~$L_1, L_2$ and~$R_1$.
  \end{enumerate}
\end{theorem}
The three items here are special cases of Proposition~\ref{prop:intgamma-taylor}, Corollary~\ref{cor:gamma-pow-xlog} and Proposition~\ref{prop:composition} below, respectively.
The coefficients~$c_1, c_2$ and~$c_*$ and the function $R_+$ are explicitly described in the precise versions below. 

The proofs of all three result are rather short, but together they allow for a simple proof of the expansion~\eqref{eq:integral-question-Iphi} in several concrete cases:
\begin{itemize}
\item In Corollary~\ref{cor:integralapplication-estermann}, we study a function~$\phi:(0, 1]\to \R^2$ having an asymptotic behaviour around~$0$ of the shape~$x^{-1/2}\abs{\log x}$. The ensuing estimate we obtain is used in~\cite[Theorem~2.1]{BettinDrappeau} to deduce a central limit theorem for central values~$\{D(1/2, x), x \in \Q\cap(0, 1]\}$ of the analytic continuation of the Estermann function
  \begin{equation}
    D(s, x) = \sum_{n \geq 1}\frac{\tau(n)}{n^s}\e^{2\pi i n x}, \qquad (\Re(s)>1),\label{eq:def-estermann}
  \end{equation}
  where~$\tau$ is the divisor function.
\item In Corollaries~\ref{cor:integralapplication-floor1x} and \ref{cor:integralapplication-largemom}, we study the functions of the shape~$\phi(x) = \floor{1/x}^\lambda$ where~$\lambda\geq 1/2$. These functions occur when studying the values~$\{\Sigma_\lambda(x), x \in \Q\cap(0, 1]\}$ of the moments of the continued fractions coefficients,
  $$ \Sigma_\lambda(x) = \sum_{j=1}^r a_j^\lambda, \qquad (x = [0; a_1, \dotsc, a_r] = \frac1{a_1 + \frac1{a_2+\dotsb}}, a_r>1), $$
  see~\cite[Theorems~2.5 and~9.4]{BettinDrappeau}. This, in turn, is applied to obtain a law of large numbers for the values of the Kashaev invariants of the~$4_1$ knot~\cite[Corollary~2.6]{BettinDrappeau}.
\item In Corollary~\ref{cor:integralapplication-dedekind}, we study the function~$\phi$ on~$(0, 1]$ given by~$\phi(x) = \floor{1/x} - \floor{1/T(x)}$, where~$T:(0, 1] \to (0, 1]$, $T(x) = \{1/x\}$ is the Gauss map. The estimate we obtain is used in~\cite[Theorem~2.7]{BettinDrappeau} to obtain an independent proof, using dynamical systems, of a theorem of Vardi~\cite{Vardi1993} on the convergence to a Cauchy law of the values of Dedekind sums.
\end{itemize}

\section{Estimation of~\eqref{eq:integral-question-Iphi} in general}

\subsection{Basic estimates}

\subsubsection{Taylor estimate}

The first and simplest method to obtain an estimate for~\eqref{eq:integral-question-Iphi} is to insert and integrate a Taylor expansion for the exponential.

\begin{proposition}\label{prop:intgamma-taylor}
  Assume that for some~$\alpha\in(0, 3]$, we have
  $$ K := \int \abs{\phi(x)}^\alpha \df\mu(x) < \infty. $$
  Then~$\phi\in\G(\alpha, 1, 1)$, and more precisely
  \begin{equation}
    \I[\phi](t) = i c_1 t + c_2 t^2 + O(K t^\alpha)\label{eq:Taylor-effective}
  \end{equation}
  with~$c_1 = \int \phi \df\mu$ if~$\alpha\geq1$, and~$c_2 = -\frac12\int \abs{\phi}^2 \df\mu$ if~$\alpha\geq2$. The implied constant is absolute.
\end{proposition}
\begin{proof}
  We use the bound~$\abs{\e^{iu} - \sum_{0\leq k < \alpha} \frac{(iu)^k}{k!}} \ll \abs{u}^\alpha$ with~$u = t\phi(x)$, and integrate over~$x$.
\end{proof}
Although it will not be useful for us here, we note that in the precise bound~\eqref{eq:Taylor-effective}, the value of~$\alpha$ could be taken as a function of~$t$. For example, if~$\mu$ is the Lebesgue measure on~$(0, 1)$ and~$\phi(x) = 1/x$, we can take~$\alpha = 1-1/\abs{\log t}$ and obtain~$\I[\phi](t) = O(t\abs{\log t})$.

\subsubsection{Using properties of the Mellin transform}

When the moment~$\int\abs{\phi}^\alpha \df \mu$ diverges at some particular~$\alpha$, we can often extract a useful expansion from the Cauchy formula and the polar behaviour of the Mellin transform. For~$x\in \R$, $s\in \C$ and~$\eta\in[0,1]$, let
$$ \phi_{s,\eta}(x) := \1_{\phi(x)\neq 0}\abs{\phi(x)}^s \exp(-s \tfrac {\pi i}2(1-\eta)\sgn\phi(x)), \qquad \phi_s(x) := \phi_{s,0}(x). $$
Note that for~$k\in\N_{>0}$,~$\phi_k(x) = (-i\phi(x))^k$. Define further
$$ G_\eta(s) := \int \phi_{s,\eta}(x) \df\mu(x). $$

\begin{proposition}\label{prop:Dirichlet-to-G}
  Let~$\alpha\in(0,3)$, $\rho\in(0, 1)$, $\delta, \eta_0> 0$ and~$\xi\in\R$. Assume that for some~$c>0$, we have
  \begin{equation}
    \int_{\phi(x)\neq 0} (\abs{\phi(x)}^c + \abs{\phi(x)}^{-c})\df\mu(x) < \infty \label{eq:gammaG-prebound}
  \end{equation}
  and that the functions~$G_\eta(s)$ for~$\eta\in[0,\eta_0]$, initially defined for~$\Re(s) \in (-c, c)$, can be analytically continued to the set
  $$ \big\{s\in\C, 0<\Re(s)\leq\alpha+\delta, s\not\in[\alpha, \alpha+\delta]\big\}. $$
  Assume further that
  \begin{equation*}
    \sup_{0\leq \eta \leq \eta_0}\int_{\substack{\tau\in\R \\ s = \alpha + \delta + i\tau}} \abs{\Gamma(-s) G_\eta(s)} \df \tau < \infty, 
  \end{equation*}
  and that there is an open neighborhood~$\cV$ of~$[\alpha, \alpha+\delta]$ for which
  \begin{equation}
    (\alpha-s)^\xi G_0(s) = \varrho + O(\abs{s-\alpha}^\rho), \qquad s\in\cV\smallsetminus[\alpha, \alpha+\delta],\ \Re(s)\leq \alpha+\delta.\label{eq:polar-bhvr}
  \end{equation}
  Then, $\phi \in \G(\alpha, \abs{\log}^{\xi-1+\upsilon_\alpha}, \abs{\log}^{\xi-1+\upsilon_\alpha-\rho})$, where $\upsilon_\alpha=1$ if $\alpha=1,2$ and $\upsilon_\alpha=0$ otherwise, and with coefficients given by
\begin{equation}\label{cfc}
     c_1 = iG_0(1) \text{ if } \alpha>1, \qquad c_2 = \tfrac12 G_0(2) \text{ if } \alpha>2, \qquad c_\ast =
    \begin{cases}
	 -\varrho/\Gamma(\xi+1), & \alpha=1,\\
	  \frac12 \varrho/\Gamma(\xi+1), & \alpha=2,\\
     \varrho \frac{\Gamma(-\alpha)}{\Gamma(\xi)}, & \alpha\notin\{1,2\}.\\
     \end{cases}
\end{equation}
\end{proposition}
\begin{proof}
  We write
  $$ \I[\phi](t) + 1 = \int \e^{it\phi(x)}\df\mu(x) = J_+ + J_- + J_0, $$
  where~$J_\pm$ corresponds to the part of the integral restricted to~$\pm\phi>0$.

  For all~$\eps\in(0,\frac{\pi}2\eta_0)$, define
  $$ J_+(\eps) := \int_{\phi(x)>0} \e^{(-\eps+i)t\phi(x)}\df\mu(x), \qquad J_-(\eps) := \int_{\phi(x)<0} \e^{(\eps+i)t\phi(x)}\df\mu(x).  $$
  By dominated convergence, we have~$J_+ := \lim_{\eps\to 0^+} J_+(\eps)$, and similarly for~$J_-$. We use the Mellin transform formula for the exponential
  $$ \e^{-y} = \frac1{2\pi i} \int_{-c/2-i\infty}^{-c/2+i\infty} \Gamma(-s) \abs{y}^s \e^{s\arg(y)} \df s $$
  valid for~$\Re(y)>0$, see~\cite[eq.~17.43.1]{GZ} (the extension to non-real~$y$ is straightforward by the Stirling formula~\cite[eq.~8.327.1]{GZ}). Inserting this in~$J_{\pm}(\eps)$, we obtain
  $$ J_+(\eps) + J_-(\eps) = \frac1{2\pi i} \int_{-c/2-i\infty}^{-c/2+i\infty} \Gamma(-s) G_\eta(s) \abs{1+i\eps}^s t^s \df s, $$
  where~$\eta = \frac2\pi \arctan\eps\leq \frac{2\eps}{\pi}\leq \eta_0$. We move the contour forward to~$\Re(s) = \alpha + \delta$. The simple pole at~$s=0$ contributes~$\int_{\phi(x)\neq 0} \df \mu(x)$, and therefore by adding the contribution from~$J_0$ we get
  \begin{align*}
    J_0 + J_+(\eps) + J_-(\eps) = {}& 1 + R + \frac1{2\pi i}\int_{H(\alpha, \alpha+\delta)} \Gamma(-s) G_\eta(s) t^s \abs{1+i\eps}^s\df s \\
    {}& \quad + \frac1{2\pi i}\int_{\Re(s) = \alpha + \delta} \Gamma(-s) G_\eta(s) t^s \abs{1+i\eps}^s \df s,
  \end{align*}
  where~$R$ consists of the contribution of the residues at~$1$ (if~$\alpha>1$) and~$2$ (if~$\alpha>2$). Here~$H(\alpha, \alpha+\delta)$ is a Hankel contour, going from~$\alpha+\delta-i0$ to~$\alpha+\delta+i0$ passing around~$\alpha$ from the left. The last integral is bounded by the triangle inequality, using our first hypothesis on~$G_\eta$, which gives
  $$ \frac1{2\pi i}\int_{\Re(s) = \alpha + \delta} \Gamma(-s) G_\eta(s) t^s \abs{1+i\eps}^s \df s \ll t^{\alpha+\delta}, $$
  uniformly in $\eps$.
  Passing to the limit~$\eps\to 0$, there remains to prove
  $$ \frac1{2\pi i}\int_{H(\alpha, \alpha+\delta)} \Gamma(-s) G_0(s) t^s \df s = c_\ast t^\alpha \abs{\log t}^{\xi-1+\upsilon_\alpha} + O(t^\alpha \abs{\log t}^{\xi-1+\upsilon_\alpha-\rho}). $$
  This is done by using our second hypothesis along with a standard Hankel contour integration argument; we refer to \textit{e.g.} Corollary II.0.18 of~\cite{Tenenbaum2015a} for the details.
\end{proof}

An important special case is the following.

\begin{corollary}\label{cor:gamma-pow-xlog}
  Let~$\mu$ be defined on~$[0, 1]$ by~$\df \mu(x) = f(x) \df x$ where~$f \in \CC^1([0, 1])$. Let~$a\in\R\smallsetminus\{0\}$. For all~$\beta>\frac13$,~$\lambda \geq0 $, and~$\phi$ given by
  $$ \phi(x)= a x^{-\beta}\abs{\log x}^\lambda $$
one has $\phi \in \G(1/\beta, \abs{\log}^{\lambda/\beta+\upsilon_{1/\beta}}, \abs{\log}^{\lambda/\beta+\upsilon_{1/\beta}-1+\eps})$ for any $\eps\in(0,1)$ and with 
\begin{equation*}
  c_\ast = f(0) \frac{\abs{a}^{1/\beta} e^{\frac{-\pi i\sgn a}{2\beta}}}{\beta^{\lambda/\beta+1}} \times
  \begin{cases}
    -(\lambda+1)^{-1}, & \beta=1,\\
    (4\lambda+2)^{-1}, & \beta=1/2,\\
    \Gamma(-1/\beta), & \beta\notin\{1,1/2\}.\\
  \end{cases}
\end{equation*}
and $c_1 = \int \phi \df\mu$ if $\beta<1$ and $c_2 = -\frac12\int \abs{\phi}^2 \df\mu$ if $\beta<\frac12$.

\end{corollary}
\begin{proof}
First, we write $\df\mu(x) = f(0)\chi(x)\df x  + x g(x) \df x$, where $\chi$ is the characteristic function of the interval $[0,1]$ and~$g\in \CC([0, 1])$. For the contribution of $\chi\df x$ we apply Proposition~\ref{prop:Dirichlet-to-G} with any fixed~$c<1/\beta$, $\alpha=1/\beta$, $\xi=\lambda/\beta+1$, any fixed~$\rho\in(0,1)$ and $\delta>0$. By~\cite[4.272.6]{GZ}, for~$\Re(s)<1/\beta$ and  $\eta\in[0,1]$ we have 
  $$ G_\eta(s)=e^{-s \tfrac {\pi i}2(1-\eta)\sgn(a)}|a|^s \int_0^1 x^{-\beta s}\abs{\log x}^{\lambda s}\df x = e^{-s \tfrac {\pi i}2(1-\eta)\sgn(a)}|a|^s  \frac{\Gamma(\lambda s+1)}{(1-\beta s)^{\lambda s + 1}}. $$
Notice also that by Stirling's formula $G_\eta(s)\ll \e^{\pi(\frac{1-\eta}2) |\tau|} |\tau|^{-1/2}$ as $|\tau| = |\Im s|\to\infty$, so that in any case $\Gamma(-s)G_\eta(s) \ll |\tau|^{-1-\Re(s)}$. Therefore the hypotheses of Proposition~\ref{prop:Dirichlet-to-G} are easily verified with 
$$ \varrho = \abs{a}^{1/\beta} e^{\frac{-\pi i\sgn a}{2\beta}}\frac{\Gamma(\lambda/\beta+1)}{\beta^{\lambda/\beta+1}}. $$
Thus,
$$
\int_0^1 (\e^{it\phi(x)}-1) \df x = it c_1'  +c_2' t^2 +c_* t^{1/\beta}\abs{\log t}^{\lambda/\beta+\upsilon_{1/\beta}} + O(t^{1/\beta}\abs{\log t}^{\lambda/\beta+\upsilon_{1/\beta}-\rho})
$$
 with coefficients as given in~\eqref{cfc} with $G_0(1)=-i\int \phi \chi \,dx$ and $G_0(2)=- \int \phi^2 \chi \,dx$.
Finally, as in Proposition~\ref{prop:intgamma-taylor} we deduce 
$$
\int (\e^{it\phi(x)}-1) x g(x) \df x = i c_1''t  +c_2'' t^2 + O(K t^{\alpha'})
$$
for any $0<\alpha'<\min(3,\frac2{\beta})$ and with $c_1''=\int \phi(x) x g(x) \df x$ if $\alpha'>1$ and $c_2''=-\frac12\int \phi(x)^2 x g(x) \df x$ if $\alpha'>2$. The result then follows.
\end{proof}

\subsection{Addition}

\begin{lemma}\label{lem:addition}
  For~$j\in\{1, 2\}$, let~$\delta_j(x) = \e^{it\phi_j(x)}-1$. Then
  \begin{align*}
    \I[\phi_1 + \phi_2](t) = {}& \I[\phi_1](t) + \I[\phi_2](t) + \int \delta_1(x)\delta_2(x)\df\mu(x) \numberthis\label{eq:split-delta12} \\
    = {}&  \I[\phi_1](t) + \I[\phi_2](t) + O\Big(\prod_{j\in\{1,2\}}\abs{\Re \I[\phi_j](t)}^{1/2}\Big)
  \end{align*}
\end{lemma}
\begin{proof}
  The first equation is simply the relation~$\e^{it(\phi_1(x) + \phi_2(x))} - 1 = \delta_1(x) + \delta_2(x) + \delta_1(x)\delta_2(x)$ integrated over~$x$. The last term is bounded using the Cauchy-Schwarz inequality
  $$ \Big(\int\abs{\delta_1(x)\delta_2(x)}\df\mu(x)\Big)^2 \leq \prod_{j\in\{1,2\}} \int \abs{\delta_j(x)}^2\df\mu(x) $$
  and expanding the square on the right-hand side.
\end{proof}

\begin{proposition}\label{prop:composition}
  For~$j\in\{1, 2\}$, let~$\alpha_j\in (0, 2]$, let~$L_j, R_j$ be positive functions defined on a neighborhood of~$0$ in~$\R_+^*$, and~$\phi_j \in \G(\alpha_j, L_j, R_j)$. If~$\alpha_1 \leq \alpha_2$, and under the following assumptions:
  \begin{itemize}
  \item $R_j(t), L_j(t) = t^{o(1)}$ as~$t\to 0$,
  \item $R_j(t) = O(L_j(t))$,
  \item $t^2 = O(t^{\alpha_1}L_1(t))$,
  \end{itemize}
  we have
  $$ \phi_1 + \phi_2 \in \G(\alpha_1, L_1, R_+), \qquad R_+ = \begin{cases} R_1 & \text{if } \alpha_1 < \alpha_2, \\ R_1 + L_2 + \sqrt{L_1L_2} & \text{if }\alpha_1=\alpha_2<2, \\ R_1 + L_2 + \sqrt{L_1}(\sqrt{L_2}+1) & \text{if }\alpha_1=\alpha_2=2. \end{cases} $$
  Moreover,
  \begin{align*}
    c_1(\phi_1 + \phi_2) = {}& c_1(\phi_1) + c_1(\phi_2), \\
    c_\ast(\phi_1 + \phi_2) = {}& c_\ast(\phi_1).
  \end{align*}
\end{proposition}

\begin{proof}
  We use Lemma~\ref{lem:addition}; when computing the real part in~\eqref{eq:split-delta12}, the term~$ic_1 t$ vanishes.
\end{proof}

\begin{remark}
  Note that using this result might induce a slight quantitative loss in the two cases when~$\alpha_1=\alpha_2$. What is gained at this price is that we are only required to study each~$\phi_j$ separately, which simplifies the analysis.

  We also remark that this estimate is useful only when the term~$c_2 t^2$ is not relevant in~\eqref{eq:expansion-Iphi}. In the complementary case, Proposition~\ref{prop:intgamma-taylor} can be used, although the ensuing error term will typically be worse than optimal by a factor of~$\abs{\log t}$.

  It is straightforward to generalize Proposition~\ref{prop:composition}, affecting to each~$\phi_j$ a different value of the frequency: under the same hypotheses and notations, and additionally that~$L_j, R_j$ tend monotonically to~$+\infty$ at~$0$,
  \begin{align*}
    \int \e^{i t_1 \phi_1(x) + it_2 \phi_2(x)}\df\mu(x) = {}& 1 + i c_1(\phi_1)t_1 + i c_1(\phi_2)t_2 + c_* t_1^{\alpha_1} L_1(t_1) + O(t_+^2 + t_+^{\alpha_1} R_+(t_+)),
  \end{align*}
  where~$c_1, c_*$ are as in the conclusion of Proposition~\ref{prop:composition}, and~$t_+ = \max\{t_1, t_2\}$.
\end{remark}

\section{Applications}\label{sec:applications}
We now describe the applications we will be interested in. The measure is the Gauss-Kuzmin distribution
$$ \df \mu(x) = \frac{\df x}{(1+x)\log 2} \qquad (x\in[0, 1]). $$
The measure~$\mu$ is invariant under the Gauss map~$T(x) = \{1/x\}$ on~$(0, 1)$, in particular,
\begin{equation}
  \I[\phi \circ T](t) = \I[\phi](t).\label{eq:integral-t-invarT}
\end{equation}

\subsection{Central values of the Estermann function}

The first application we discuss is the ``period function''~$\phi: \R \to \C$ associated with the Estermann function~\eqref{eq:def-estermann}, namely
$$ \phi(x) = D(\tfrac12, 1/x) - D(\tfrac12, x), $$
initially defined in~$\Q \cap (0, 1]$. By~\cite{Bettin2016}, this function can be extended to a continuous function on~$(0, 1]$, more precisely given by an expression of the shape~\eqref{eq:vphi-estermann-explicit} below. Interpreting~$\phi$ to be~$\R^2$-valued, the analogue of the integral~\eqref{eq:integral-question-Iphi} is estimated using the following.

\begin{corollary}\label{cor:integralapplication-estermann}
  Let~$\eps>0$,~$\cE:[0,1]\to\C$ be a bounded, continuous function, and
  \begin{equation}
    \begin{aligned}
      \vphi_j(x) := {}&
      \begin{pmatrix}
        \frac12 x^{-1/2}\big(\log(1/x) + \gamma_0 - \log(8\pi) - \frac\pi2 \big) + \zeta(\tfrac12)^2 + \Re\cE((-1)^jx) \\
        \tfrac{(-1)^{j-1}}2 x^{-1/2}\big(\log(1/x) + \gamma_0 - \log(8\pi) + \frac\pi2 \big) + \Im\cE((-1)^jx)
      \end{pmatrix}.
    \end{aligned}\label{eq:vphi-estermann-explicit}
\end{equation}
  Let also~$\bm{u}_j := \big(\begin{smallmatrix} 1 \\ (-1)^{j-1} \end{smallmatrix} \big)$. Then for some vector~${\bm\mu}\in\R^2$, and all~$\vt \in\R^2$, we have
  \begin{align*}
    \int_0^1 & \e^{i\langle \vt, \vphi_{1}(x) + \vphi_{2}(T(x))\rangle} \df\mu(x) \\
    {}& = 1 + i\langle \vt, {\bm \mu}\rangle - \frac1{3\log 2}\sum_{j\in\{1,2\}} \langle\vt, \bm{u}_j\rangle^2 \abs{\log\abs{\langle\vt, \bm{u}_j \rangle}}^3 + O_\eps(\nvt^2 \abs{\log \nvt}^{2+\eps}).
  \end{align*}
\end{corollary}
\begin{proof}
  Let~$\eps\in(0, 1)$. Using Corollary~\ref{cor:gamma-pow-xlog} with~$\beta=1/2$ and~$\lambda\in\{0, 1\}$, and Proposition~\ref{prop:intgamma-taylor}, we obtain
  \begin{align*}
    (x \mapsto \pm \tfrac12 x^{-1/2}\abs{\log x}) \in \G(2, \abs{\log}^3, \abs{\log}^{2+\eps}), \\
    (x \mapsto (\gamma_0 - \log(8\pi) + \tfrac\pi 2) x^{-1/2}) \in \G(2, \abs{\log},\abs{\log}^\eps), \\
    (x \mapsto \Im\cE(\pm x)) \in \G(3, 1, 1),
  \end{align*}
  as well as~$c_\ast(x \mapsto \pm \tfrac12 x^{-1/2}\abs{\log x}) = -\frac{1}{3\log 2}$.
  From Proposition~\ref{prop:composition} and the ensuing remark, and using the property~\eqref{eq:integral-t-invarT}, we obtain for~$j\in\{1, 2\}$
  $$ \int_0^1 (\e^{i\langle{\vt,\vphi_j(x)}\rangle}-1)\df\mu(x) = i\langle\vt, \bm{\mu}_j\rangle + c_\ast \langle\vt, \bm{u}_j\rangle^2 \abs{\log \abs{\langle\vt, \bm{u}_j\rangle}}^3 + O_\eps(\nvt^2\abs{\log \nvt}^{2+\eps}), $$
  where~$\bm{\mu}_1, \bm{\mu}_2 \in \R^2$. On the other hand, we have
  $$ \Delta(\vt) := \int_0^1 (\e^{i\langle\vt, \vphi_{1}(x)\rangle}-1)(\e^{i\langle{\vt,\vphi_{2}(T(x))}\rangle}-1)\df\mu(x) = \int_0^1 (\e^{i\langle{\vt,\vphi_{2}(x)}\rangle}-1) F_x(\vt) \df x, $$
  where
  $$ F_x(\vt) = \frac1{\log 2}\sum_{n\geq 1} \frac{\e^{i\langle{\vt,\vphi_{1}(1/(n+x))}\rangle}-1}{(n+x)(n+x+1)}. $$
  By a Taylor expansion at order 1, we have~$\abs{F_x(\vt)}\ll \nvt$ uniformly in~$x$, and therefore
  $$ \abs{\Delta(\vt)} \ll \nvt^2 \int_0^1 \norm{\vphi_2(x)}\df x \ll \nvt^2. $$
  By~\eqref{eq:split-delta12}, we deduce
  $$ \int_0^1 \e^{i\langle \vt, \vphi_{1}(x) + \vphi_{2}(T(x))\rangle} \df\mu(x) = 1 + \int_0^1(\e^{i\langle{\vt,\vphi_{1}(x)}\rangle}+\e^{i\langle{\vt,\vphi_{2}(T(x))}\rangle}-2)\df\mu(x) + O(\nvt^2), $$
  whence the claimed estimate.
\end{proof}

\subsection{Moments of continued fractions coefficients}

The next application we consider pertains to the moments functions~$\Sigma_\lambda$ of continued fractions coefficients, where~$\lambda\geq 0$ is the order of the moment. The function of interest to us here is
$$ \phi_\lambda(x) = \floor{1/x}^\lambda. $$
The case~$\lambda<1/2$ can be easily dealt with using Proposition~\ref{prop:intgamma-taylor}, so we do not focus on it here.

A first approach is to use Proposition~\ref{prop:composition} to approximate~$\floor{1/x}$ by~$1/x$, and then use Corollary~\ref{cor:gamma-pow-xlog}. This leads to the following.

\begin{corollary}\label{cor:integralapplication-largemom}
  Let~$\lambda\geq 1/2$. The function~$\phi_\lambda$ given by~$\phi_\lambda(x) = \floor{1/x}^\lambda$ satisfies the following.
  \begin{itemize}
    \item If~$\lambda=1/2$, then with~$c_\ast = -1/(\log 2)$, we have
    \begin{equation}
      \I[\phi_{1/2}](t) = i c_1t + c_\ast t^2 \abs{\log t} + O_\eps(t^2 \abs{\log t}^\eps).\label{eq:appl-largemom-1/2}
    \end{equation}
    \item If~$\lambda>1/2$ and~$\lambda\neq 1$, then with~$c_\ast = -\exp(-\pi i/(2\lambda))\Gamma(1-1/\lambda)/\log 2$, we have
    \begin{align*}
      \I[\phi_\lambda](t) = (\1_{\lambda<1})ic_1t + c_\ast t^{1/\lambda} + O_\eps(t^{1/\lambda}\abs{\log t}^{-1+\eps})
    \end{align*}
  \end{itemize}
  When~$1/2\leq\lambda<1$, we have~$c_1 = \int_0^1 \phi_\lambda(x)\df\mu(x)$.
\end{corollary}

\begin{proof}
  We write~$\phi_\lambda(x) = p_\lambda(x) + r_\lambda(x)$, where~$p_\lambda(x) = x^{-\lambda}$ and $r_\lambda(x) \ll_\lambda \floor{1/x}^{\lambda-1}$. By Proposition~\ref{prop:intgamma-taylor}, we have~$r_\lambda \in \G(\min(3, \frac1{\lambda-1/3}), 1, 1)$.

  We consider first the case~$\lambda>1/2$,~$\lambda\neq 1$. By Corollary~\ref{cor:gamma-pow-xlog}, we have~$p_\lambda \in \G(\frac1\lambda, 1, \abs{\log}^{-1+\eps})$. We deduce, by Proposition~\ref{prop:composition}, that~$\phi_\lambda\in\G(\frac1\lambda, 1, \abs{\log}^{-1+\eps})$, and this yields the second and third cases.
  
  If~$\lambda=1/2$, then Corollary~\ref{cor:gamma-pow-xlog} implies~$p_{1/2}\in\G(2, \abs{\log },\abs{\log}^{\eps})$, and by Proposition~\ref{prop:intgamma-taylor}, for some~$c\in \R$, we have
  $$ \I[r_{1/2}](t) = i c t + O(t^2) $$
  On the other hand, since~$\abs{(\e^{it p_{1/2}(x)}-1)(\e^{it r_{1/2}(x)}-1)} \ll t^2 \abs{p_{1/2}(x)r_{1/2}(x)} \ll t^2$, we get
  $$ \int_0^1 (\e^{itp_{1/2}(x)}-1)(\e^{itr_{1/2}(x)}-1)\df\mu(x) = O(t^2). $$
  By~\eqref{eq:split-delta12}, we conclude~\eqref{eq:appl-largemom-1/2} as claimed.
\end{proof}

The case~$\lambda=1$ could be analyzed by the same method, but we chose to study it separately to obtain a more precise error term by another approach, using Proposition~\ref{prop:Dirichlet-to-G} directly. The associated Mellin transform~$G_0(s)$ is related to the Riemann~$\zeta$-function.

\begin{corollary}\label{cor:integralapplication-floor1x}
  The function~$\phi$ given by~$\phi(x) = \floor{1/x}$ satisfies
  $$ \I[\phi](t) = -\tfrac{it}{\log 2}(\log t + \gamma_0 - \tfrac{\pi i}2) + O_\eps(t^{2-\eps}). $$
\end{corollary}

\begin{proof}
  The integral~\eqref{eq:gammaG-prebound} converges for all~$c<1$. A quick computation shows that an analytic continuation of~$G_\eta(s)$ is given by
  $$ G_\eta(s) = \frac{\exp(-s \frac {\pi i}2(1 - \eta))}{\log 2}\big\{\zeta(2-s) + H(s)\big\}, $$
  where~$H(s) = \sum_{n\geq 1} n^s(\log(1+\frac1{n(n+2)}) - \frac1{n^2})$ is analytic and uniformly bounded in~$\Re(s)\leq 2-\eps$. We have
  $$ \int_{\Re(s)=2-\eps} \abs{\Gamma(-s)G_\eta(s)}\abs{\df s} \ll_\eps 1 + \int_0^\infty \abs{\zeta(\eps+i\tau)} \frac{\df \tau}{1+\tau^2} \ll_\eps 1 $$
  by the Stirling formula. The polar behaviour~\eqref{eq:polar-bhvr} is given by
  $$ G_0(s) = \frac{\exp(-s\frac {\pi i}2)}{\log 2}\big\{\zeta(2-s) + H(s)\big\} = \frac{\exp(-s\frac {\pi i}2)}{\log 2}\Big\{\frac1{1-s} + A + O(s-1)\Big\} $$
  for~$s$ in a neighborhood of~$1$, where
  \begin{align*}
    A = {}& \sum_{n\geq 1}\Big(n\log\Big(1+\frac1{n(n+2)}\Big) -\log\Big(1+\frac1n\Big)\Big) \\
    = {}& -\lim_{N\to\infty} \sum_{n=1}^N \Big(n\log\Big(1+\frac1{n+1}\Big) - (n-1) \log\Big(1+\frac1n\Big)\Big) \\
    = {}& -1.
  \end{align*}
  Applying Proposition~\ref{prop:Dirichlet-to-G} with~$\delta = 1/2$ and~$\alpha=1$ yields the claimed result up to~$O(t)$. Our more precise statement follows from noting that there is no branch cut along~$s\geq 1$ in this case, so that the residue theorem may be used. We obtain
  $$ \underset{s=1}\Res\ \Gamma(-s)G_0(s)t^s = \tfrac{it}{\log 2}(\gamma_0-\tfrac{\pi i}{2} + \log t), $$
  whence the claimed estimate. One could go further, isolating a pole of order~$2$ at~$s=2$, and this would give an error term~$O(t^2\abs{\log t})$.
\end{proof}

\subsection{Dedekind sums}

The final example we discuss is related to Dedekind sums, for the definition of which we refer to~\cite[Section~2.4]{BettinDrappeau}. The ``period function''~$\phi$ relevant to us here is
$$ \phi(x) = \floor{1/x} - \floor{1/T(x)}. $$
Compared with the case of~$x\mapsto \floor{1/x}$ studied in Corollary~\ref{cor:integralapplication-floor1x}, the relevant exponent~$\alpha$ is again~$1$, but the leading term turns out to be~$t$ (the terms~$t\log t$ vanish).

\begin{corollary}\label{cor:integralapplication-dedekind}
  The map~$\phi$ on~$(0, 1)$ given by~$\phi(x) = \floor{1/x} - \floor{1/T(x)}$ satisfies
  $$ \I[\phi](t) = - \frac{\pi}{\log 2}t + O(t^2\abs{\log t}^2). $$
\end{corollary}
\begin{proof}
  We consider
  \begin{align*}
    \Delta(t) := {}& \int_0^1 (\e^{-it\floor{1/T(x)}} - 1)(\e^{it\floor{1/x}}-1) \df\mu(x) \\
    = {}& \int_0^1 (\e^{-it\floor{1/x}} - 1)F_x(t)\df x,
  \end{align*}
  with~$F_x(t) = \frac1{\log 2} \sum_{n\geq 1} \frac{\e^{itn}-1}{(n+x)(n+1+x)}$.
  Since~$\abs{\e^{iu}-1} \ll \abs{u}^{1-1/\abs{\log t}}$ for all~$u\in\R$, we find
  $$ F_x(t) \ll t \sum_{n\geq 1} \frac 1{n^{1+1/\abs{\log t}}} \ll t \abs{\log t}. $$
  Similarly,
  $$ \int_0^1 \abs{\e^{-it\floor{1/x}} - 1}\df x \ll t \int_0^1 x^{-1+1/\abs{\log t}}\df x \ll t\abs{\log t}. $$
  We thus obtain~$\Delta(t) = O((t\log t)^2)$. Using Corollary~\ref{cor:integralapplication-floor1x} with the improved error term~$O(t^2 \abs{\log t})$, \eqref{eq:integral-t-invarT} and~\eqref{eq:split-delta12}, we deduce
  $$ \int_0^1 \e^{it(\floor{1/x} - \floor{1/T(x)})} \df\mu(x) = 1 + 2\Re I(t) + O((t\log t)^2), $$
  where~$I(t) = \int_0^1 (\e^{it\floor{1/x}}-1)\df\mu(x)$. Corollary~\ref{cor:integralapplication-floor1x} allows us to conclude.
\end{proof}

\bibliographystyle{../../amsalpha2}
\bibliography{../../bib2}

\end{document}